\documentclass[11pt]{amsart}
\usepackage[latin1]{inputenc}
\usepackage[all]{xy}
\usepackage{graphics}
\usepackage{amscd}
\usepackage{amssymb}
\usepackage{amsthm}
\usepackage{enumitem}
\usepackage{comment}
\newtheorem{thm}{Theorem}[section]

\newtheorem{pro}[thm]{Proposition}
\newtheorem{lem}[thm]{Lemma}
\newtheorem{cor}[thm]{Corollary}
\newtheorem{defn}[thm]{Definition}

\newtheorem*{rem*}{Remarks}
\newtheorem{rems}[thm]{Remark}
\newtheorem*{conj*}{Conjecture}

\DeclareMathOperator{\C}{\mathbb{C}}

\DeclareMathOperator{\Q}{\mathbb{Q}}
\DeclareMathOperator{\R}{\mathbb{R}}
\DeclareMathOperator{\Z}{\mathbb{Z}}

\DeclareMathOperator{\Spec}{Spec}

\newcommand{\mrm}{\mathrm}
\newcommand{\mbb}{\mathbb}

\title{On the residue of Eisenstein classes of Siegel varieties}
\author{Francesco Lemma}
\address{Institut math\'ematique de Jussieu-Paris Rive Gauche, UMR 7586, B\^atiment Sophie Germain, Case 7012, 75205 Paris Cedex 13}
\email{francesco.lemma@imj-prg.fr}

\begin{document}

\begin{abstract} Eisenstein classes of Siegel varieties are motivic cohomology classes defined as pull-backs by torsion sections of the polylogarithm prosheaf on the universal abelian scheme. By reduction to the Hilbert-Blumenthal case, we prove that the Betti realization of these classes on Siegel varieties of arbitrary genus have non-trivial residue on zero dimensional strata of the Baily-Borel compactification. A direct corollary is the non-vanishing of a higher regulator map.
\end{abstract}

\maketitle

\tableofcontents

\section{Introduction} 

The explicit construction of motivic cohomology classes, and the computation of their images under regulator maps, is one of the main ingredients of most of the proofs of  the conjectures of Beilinson and Bloch-Kato on special values of $L$-functions. Polylogarithms, which have been defined by Beilinson-Levin for relative curves \cite{beilinson-levin1}, \cite{beilinson-levin2} and by Wildeshaus for abelian schemes \cite{wildeshaus}, are one interesting source of such cohomology classes. For abelian schemes, the polylogarithm is a prosheaf on the complement of the zero section. By pulling back the polylogarithm along a non-zero torsion section, one gets some cohomology classes, the so-called Eisenstein classes, on the base of the abelian scheme. In the elliptic case, these classes have been studied intensively. In Deligne cohomology, they have been explicitly described by some real analytic Eisenstein series by Beilinson-Levin \cite{beilinson-levin1}; in \'etale cohomology they are closely related to the Kato-Siegel units giving rise to Kato's Euler system, as shown by Kings \cite{kings3}; in syntomic cohomology, they are described by Katz's measure \cite{bannai-kings}. However, very little is known about the Eisenstein classes for abelian schemes of dimension $> 1$. Blotti\`ere showed in his thesis \cite{blottiere1} that Levin's currents described the topological realization of the polylogarithm of abelian schemes. He also deduced from this result that, in the Hilbert-Blumenthal case, the residue of the Eisenstein classes at the cusps of the Baily-Borel compactification are described by special values of the zeta function of the underlying totally real field \cite{blottiere}.\\

In this article, we consider Eisenstein classes of Siegel varieties of arbitrary genus $g$. The boundary of the Baily-Borel compactification of these Shimura varieties is stratified by Siegel varieties of genus $0 \leq g' \leq g-1$. We show that there exists a torsion section of the universal abelian scheme such that the associated Eisenstein class has non-zero residue at the zero dimensional strata of the boundary. This is the first non-vanishing result for Eisenstein classes of Siegel varieties of genus $g>1$. The proof does not rely on a complicated residue computation but rather on the compatibility with base change of the polylogarithm prosheaf, which implies a similar property for the Eisenstein classes. The idea is that by embedding a Hilbert-Blumenthal variety on the considered Siegel variety, the non-vanishing of the residue is reduced to a consequence of Blotti\`ere's result. We would like to mention that arithmetic Eisenstein classes on the Siegel space have been considered by Faltings \cite{faltings}.\\

\textbf{Aknowledgements.} It is a pleasure to thank  David Blotti\`ere for his interest and J\"org Wildeshaus for correspondence which gave rise to this work.

\section{Eisenstein classes}

The setting is the following. The schemes that we consider are $\C$-schemes. If $X$ is such a scheme, we work with the abelian category of sheaves of $\Q$-modules $\mrm{Sh}(X)$ on the usual analytic topology on the analytification $X^{an}$ and the full subcategory $D^b_c(X, L)$ of its derived category whose objects are complexes whose cohomology is constructible. There is a full 6 functor formalism on the categories $D^b_c(X, L)$, see for example \cite{kashiwara-schapira}. Cohomology is Betti cohomology.\\

Let $\pi: A \rightarrow S$ be an abelian scheme with zero section $e: S \rightarrow A$. Let us denote by $\mathcal{H}$ the local system $\mathcal{H}=\underline{\mathrm{Hom}}(R^1\pi_*\Q, \Q)$. 

\subsection{The logarithm prosheaf} By the Leray spectral sequence for the composite functor
$$
\mrm{RHom}_{\mrm{Sh}(S)}(\Q(0),\,\,\,) \circ \pi_*
$$
applied to $\pi^* \mathcal{H}$, we have the exact sequence
$$
0 \rightarrow \mathrm{Ext}^1_{S}(\Q, \mathcal{H}) \rightarrow \mathrm{Ext}^1_{A}(\Q, \pi^*\mathcal{H}) \rightarrow \mathrm{Hom}_{S}(\Q, R^1\pi_*\pi^*\mathcal{H}) \rightarrow 0.
$$
Note that this exact sequence is split by $e^*$ and that the right hand term is isomorphic to $\mathrm{Hom}_{S}(\mathcal{H}, \mathcal{H})$ because of the projection formula $R^1\pi_*\pi^*\mathcal{H} \simeq R^1\pi_*\Q \otimes \mathcal{H}.$ Let us denote by $\delta: \mathrm{Ext}^1_{S}(\Q, \pi^*\mathcal{H}) \rightarrow \mathrm{Hom}_{S}(\mathcal{H}, \mathcal{H})$ the right hand morphism of the above exact sequence.

\begin{pro} \label{carac-log} Up to isomorphism in $\mrm{Sh}(A)$, there exists a unique local system $\mathrm{Log}^{(1)}_A$ which is an extension
$$
0 \rightarrow \pi^*\mathcal{H} \rightarrow \mathrm{Log}^{(1)}_A \rightarrow \Q \rightarrow 0
$$
and whose equivalence class $[\mathrm{Log}^{(1)}_A] \in \mathrm{Ext}^1_{A}(\Q(0), \pi^*\mathcal{H})$ is such that
$e^*[\mathrm{Log}^{(1)}_A]=0$ and $\delta[\mathrm{Log}^{(1)}_A]=\mrm{id}_{\mathcal{H}}$.
\end{pro}

\begin{proof} Trivial.
\end{proof}

\begin{rems} There exists a variation of Hodge structure whose underlying local system is $\mathrm{Log}^{(1)}_A$ and which is characterised up tu \textit{unique} isomorphism by properties analogous to the ones given in the above result (see \cite{blottiere1} 3.1 and 3.2). We will not need this fact in this paper.
\end{rems}

\begin{defn} \label{def-log} Let $\mathrm{Log}^{(k)}_A$ be the local system $\mathrm{Log}^{(k)}_A=\mathrm{Sym}^k \mathrm{Log}^{(1)}_A$.The logarithm sheaf is the pro-local system
$$
\mathrm{Log}_A=\underleftarrow{\lim} \mathrm{Log}^{(k)}_A
$$
where the transition maps are induced by the map $\mathrm{Log}^{(1)}_A \rightarrow \Q$. In particular, one has exact sequences
$$
0 \rightarrow \pi^* \mathrm{Sym}^k\mathcal{H} \rightarrow \mathrm{Log}^{(k)}_A \rightarrow \mathrm{Log}^{(k-1)}_A \rightarrow 0
$$
and a splitting induced by $s: e^*\Q \rightarrow e^* \mathrm{Log}^{(1)}_A$ given by
$$
e^*\mathrm{Log}_A \simeq \prod_{k \geq 0} \mathrm{Sym}^k \mathcal{H}.
$$
\end{defn}

\begin{rems} The more conceptual definition of the prosheaf $\mathrm{Log}_A$ via a universal property (see \cite{wildeshaus} I. Thm. 3.3 and Thm. 3.5) coincides with ours according to \cite{blottiere1} Prop. 3.13.
\end{rems}

\begin{lem} \label{splitting}  For every torsion section $t: S \rightarrow A$, one gets a canonical isomorphism
$$
t^*\mathrm{Log}_A \simeq \prod_{k \geq 0} \mathrm{Sym}^k \mathcal{H}.
$$
\end{lem}

\begin{proof}
Let $N$ be the order of $t$ and let $[N]: A \rightarrow A$ be the multiplication by $N$. For any $k$, we have $[N]^*\pi^* \mathrm{Sym}^k\mathcal{H}=\pi^* \mathrm{Sym}^k\mathcal{H}$ so, by induction $[N]^*\mathrm{Log}_{A} \simeq \mathrm{Log}_{A}$. As a consequence 
$$
t^*\mathrm{Log}_A \simeq t^*[N]^*\mathrm{Log}_A \simeq e^* \mathrm{Log}_A \simeq \prod_{k \geq 0} \mathrm{Sym}^k \mathcal{H}.
$$
\end{proof}

Let $j: U=A-e(S) \rightarrow A$ be the complement of the zero section in $A$. Denote by $\pi_U: U \rightarrow S$ the projection $\pi_U= \pi \circ j$ and by $\mathrm{Log}_U$ the restriction $j^* \mathrm{Log}_A$. For any integer $m$, we denote as usual by $\Q(m)$ the constant Tate local system $\Q(m)=(2i\pi)^m \Q$.

\begin{thm} \label{images-directes} (i) One has $R^i\pi_*\mathrm{Log}_A=0$ for $i \neq 2g$ and the augmentation map $\mathrm{Log}_A \rightarrow \Q$ induces canonical isomorphisms
$$
R^{2g}\pi_*\mathrm{Log}_A \simeq R^{2g}\pi_*\Q \simeq \Q(-g).
$$
(ii)  One has $R^i \pi_{U\,*} \mathrm{Log}_U=0$ for $i \neq 2g-1$ and 
$$
R^{2g-1} \pi_{U\,*} \mathrm{Log}_U \simeq \prod_{k>0} \mathrm{Sym}^k \mathcal{H}(-g).
$$
\end{thm}

\begin{proof} The first statement follows from \cite{wildeshaus} I. Cor. 4.4 p. 70. The second follows from \cite{kings1} Prop. 1.1.3.
\end{proof}

\begin{cor}\label{corollaire1} (\cite{kings1} Cor. 1.1.4)  The edge morphism in the Leray spectral sequence for $R\pi_{U\,*}$ induces a canonical isomorphism
$$
\mathrm{Ext}^{2g-1}_U(\pi_U^*\mathcal{H}, \mathrm{Log}_U(g)) \simeq \mathrm{Hom}_S(\mathcal{H}, \prod_{k>0} \mathrm{Sym}^k \mathcal{H})
$$
where the $\mrm{Ext}^{2g-1}_U$ denotes extensions in the abelian category $\mrm{Sh}(U)$.
\end{cor}

\subsection{The polylogarithm prosheaf and its torsion sections}

\begin{defn}  The polylogarithm $\mathrm{Pol}_A$ is the extension class
$$
\mathrm{Pol}_A \in \mathrm{Ext}^{2g-1}_U(\pi_U^*\mathcal{H}, \mathrm{Log}_U(g))
$$
which corresponds to the map in $\prod_{k>0} \mathrm{Hom}_S(\mathcal{H}, \mathrm{Sym}^k \mathcal{H})$ which is the identity for $k=1$ and which is zero for $k>1$, under the isomorphism of Cor. \ref{corollaire1}.
\end{defn}

To define the Eisenstein classes, we would like to describe, for any integer $k \geq 0$ and  for any non-zero torsion section $t: S \rightarrow A$, a natural map
$$
\mathrm{Ext}^{2g-1}_U(\pi_U^*\mathcal{H}, \mathrm{Log}_U(g)) \rightarrow H^{2g-1}(S, \mathrm{Sym}^k \mathcal{H}(g)).
$$
Let $k \geq 0$ be an integer and let $t: S \rightarrow A$ be a non-zero torsion section. According to Lem. \ref{splitting}, the pull-back by $t$ induces a map
$$
t^*: \mathrm{Ext}^{2g-1}_U(\pi_U^*\mathcal{H}, \mathrm{Log}_U(g)) \rightarrow \mathrm{Ext}^{2g-1}_S(\mathcal{H}, \prod_{k \geq 0} \mathrm{Sym}^k \mathcal{H}(g)).
$$
The right hand term maps naturally to $H^{2g-1}(S, \mathrm{Sym}^k \mathcal{H}(g))$ by the composition of the canonical isomorphism 
$$
\mathrm{Ext}^{2g-1}_S(\mathcal{H}, \prod_{k \geq 0} \mathrm{Sym}^k \mathcal{H}(g)) \simeq H^{2g-1}(S, \prod_{k \geq 0} \mathrm{Sym}^k \mathcal{H}(g) \otimes \mathcal{H}^\vee)
$$ 
where $\mathcal{H}^\vee$ is the dual local system, of the map induced by the $(k+1)$-th natural projection
$$
H^{2g-1}(S, \prod_{k \geq 0} \mathrm{Sym}^k \mathcal{H}(g) \otimes \mathcal{H}^\vee) \rightarrow  H^{2g-1}(S, \mathrm{Sym}^{k+1} \mathcal{H}(g) \otimes \mathcal{H}^\vee)
$$
and of the map induced by the contraction $\mathrm{Sym}^{k+1} \mathcal{H}(g) \otimes \mathcal{H}^\vee \rightarrow \mathrm{Sym}^{k} \mathcal{H}(g)$.

\begin{defn} \label{eis-ab} For any torsion section $t: S \rightarrow A$, the $k$-th Eisenstein class
$$
\mathrm{Eis}^k(t) \in H^{2g-1}(S, \mathrm{Sym}^k \mathcal{H}(g))
$$
associated to $t$ is the image of $\mathrm{Pol}_A$ under the map described above.
\end{defn}

\subsection{Base change} Let us consider the compatibility of the polylogarithm prosheaf and of the Eisenstein classes with base change. Let $f: S' \rightarrow S$ be a morphism and let $\pi': A' \rightarrow S'$ be the pull-back of $\pi: A \rightarrow S$ by $f$. We have a cartesian square
$$
\begin{CD}
A' @>f' >> A \\
@V\pi'VV         @V\pi VV \\
S' @>f>> S.
\end{CD}
$$
Let $\mathcal{H}'=\underline{\mathrm{Hom}}_{S'}(R^1\pi'_*\Q, \Q).$ By Poincar\'e duality, we have $$\mathcal{H}=\underline{\mathrm{Hom}}_S(R^1\pi_*\Q, \Q)=R^1\pi_!\Q=R^1\pi_*\Q$$
and similarly for $\mathcal{H}'.$ Hence, the proper base change theorem implies $f^*\mathcal{H} \simeq \mathcal{H}'$.

\begin{pro} \label{chgt-de-base-log} With the notations above, we have a canonical isomorphism
$$
f'^*\mathrm{Log}_A \simeq \mathrm{Log}_{A'}.
$$
\end{pro}

\begin{proof} The image of $\mathrm{Log}_A^{(1)}$ under the natural map
$$
f'^*: \mathrm{Ext}^{1}_A(\Q, \pi^*\mathcal{H}) \rightarrow \mathrm{Ext}^1_{A'}(\Q, \pi'^*\mathcal{H}').
$$
satisfies the properties characterising $\mathrm{Log}_{A'}^{(1)}$ (see Prop. \ref{carac-log}), hence $f'^*\mathrm{Log}_A^{(1)} \simeq \mathrm{Log}_{A'}^{(1)}$. This implies the statement.
\end{proof}

By the previous proposition, we have a map
$$
f'^*: \mathrm{Ext}^{2g-1}_U(\pi_{U}^*\mathcal{H}, \mathrm{Log}_U(g)) \rightarrow \mathrm{Ext}^{2g-1}_{U'}(\pi_{U'}^*\mathcal{H}', \mathrm{Log}_{U'}(g)).
$$

\begin{pro} \label{chgt-de-base-polylog} We have $$f'^*\mathrm{Pol}_A=\mathrm{Pol}_{A'}.$$
\end{pro}

\begin{proof} Denote by $j': U'=A'-e'(S') \rightarrow A'$ the open embedding complementary to the zero section $e': S' \rightarrow A'$ and let $\mathrm{Log}_{U'}=j'^* \mathrm{Log}_{A'}$. The statement follows from the fact that we have a commutative diagram
$$
\begin{CD}
\mathrm{Ext}^{2g-1}_U(\pi_{U}^*\mathcal{H}, \mathrm{Log}_U(g)) @>\sim>> \prod_{k>0} \mathrm{Hom}_S(\mathcal{H}, \mathrm{Sym}^k \mathcal{H})\\
@Vf'^*VV                                                                                                                    @Vf^*VV\\
\mathrm{Ext}^{2g-1}_{U'}(\pi_{U'}^*\mathcal{H}', \mathrm{Log}_{U'}(g)) @>\sim>>\prod_{k>0} \mathrm{Hom}_{S'}(\mathcal{H}', \mathrm{Sym}^k \mathcal{H}').
\end{CD}
$$
where the right hand vertical map sends the map in $\prod_{k>0} \mathrm{Hom}_S(\mathcal{H}, \mathrm{Sym}^k \mathcal{H})$ which is the identity for $k=1$ and which is zero for $k>1$ to the analogous map in the lower right hand corner of the diagram.
\end{proof}

\begin{cor} \label{chgt-de-base-eisenstein}Let $t: S \rightarrow A$ be a non-zero torsion section and let $t': S' \rightarrow A'$ be the pull-back of $t$ by $f'$. Then
$$
f^*\mathrm{Eis}^k(t)=\mathrm{Eis}^k(t').
$$
\end{cor}

\begin{proof} The statement follows from the commutativity of the following diagram:
$$
\begin{CD}
\mathrm{Ext}^{2g-1}_U(\pi_{U}^*\mathcal{H}, \mathrm{Log}_U(g))  @>f'^*>> \mathrm{Ext}^{2g-1}_{U'}(\pi_{U'}^*\mathcal{H}', \mathrm{Log}_{U'}(g)) \\
@Vt^*VV                                                                                                                             @Vt'^*VV\\
\mathrm{Ext}^{2g-1}_S(\mathcal{H}, \prod_{k \geq 0} \mathrm{Sym}^k \mathcal{H}(g)) @>f^*>> \mathrm{Ext}^{2g-1}_{S'}(\mathcal{H}', \prod_{k \geq 0} \mathrm{Sym}^k \mathcal{H}'(g))\\
@VVV                                                                                                                                    @VVV\\
H^{2g-1}(S, \mathrm{Sym}^k \mathcal{H}(g)) @>f^*>> H^{2g-1}(S', \mathrm{Sym}^k \mathcal{H}'(g)).
\end{CD}
$$
\end{proof}

\section{Geometry}

\subsection{Shimura data and Shimura varieties} We will need some results of Pink's thesis \cite{pink} on the functoriality of the Baily-Borel compactification of Shimura varieties. Pink works in the setting of Shimura data that we shall recall now. Let $\mbb{S}=\mrm{Res}_{\C/\R} \mbb{G}_{m \C}$ be the Deligne torus. Let $P$ be a connected linear algebraic group over $\Q$, let $W$ be its unipotent radical and let $U$ be a subgroup of $W$. A mixed Shimura datum with underlying group $P$ is a triple $(P, \mathfrak{X}, h)$ where $\mathfrak{X}$ is a left homogeneous space under the subgroup $P(\R) U(\C) \subset P(\C)$ and where $h: \mathfrak{X} \rightarrow \mrm{Hom}(\mbb{S}_{\C}, P_{\C})$ is a $P(\R) U(\C)$-equivariant map such that the properties (i)-(viii) of \cite{pink} Def. 2.1 are satisfied. The Shimura datum is said pure if the group $P$ is reductive.

\begin{defn} \label{morphism-Shimura} A morphism $(P_1, \mathfrak{X}_1, h_1) \rightarrow (P_2, \mathfrak{X}_2, h_2)$ of mixed Shimura data is a pair $(\phi, \psi)$ where $\phi: P_1 \rightarrow P_2$ is a morphism of algebraic groups and $\psi: \mathfrak{X}_1 \rightarrow \mathfrak{X}_2$ is a $P_1(\R) U_1(\C)$-equivariant map such that the following diagram commutes
$$
\begin{CD}
\mathfrak{X}_1 @>\psi >>                                                                                     \mathfrak{X}_2\\
@VVV                                                                                                                                       @VVV\\
\mrm{Hom}(\mbb{S}_{\C}, P_{1 \C})      @>h \mapsto \phi \circ h >>                                  \mrm{Hom}(\mbb{S}_{\C}, P_{2 \C}).\\
\end{CD}
$$
\end{defn}

\subsection{Siegel varieties} \label{siegel-var}Let $g \geq 1$ and $n \geq 3$ be two integers. We consider the functor which associates to a scheme $T$ over $\Spec \Q$ the set of isomorphism classes of triples $(A, \lambda, \eta)$ where $A/T$ is an abelian scheme of relative dimension $g$, $\lambda$ is a principal polarization of $A/T$, and $\eta: A[n] \simeq (\Z/n\Z)_S^{2g}$ is a principal level $n$ structure compatible with $\lambda$ (see \cite{laumon} Partie I.3, which can be adapted from $g=2$ to arbitrary $g$, for details). This functor is represented by a smooth scheme $\mathcal{S}$ of finite type over $\Spec \Q$. The variety $\mathcal{S}$ is the Siegel variety of genus $g$ and level $n$. Let $\mathcal{A}/\mathcal{S}$ be the universal abelian scheme. We will need the description of $\mathcal{S}$ as a Shimura variety. So, let $I_g$ be the identity matrix of size $g$ and let $\psi$ be the symplectic form on $\Z^{2g}$ whose matrix in the canonical basis is
\begin{equation} \label{symplectic}
\psi=
\begin{pmatrix}
 & I_g\\
-I_g & \\
\end{pmatrix}.
\end{equation}
The symplectic group $G=\mathrm{GSp}(2g)$ is the algebraic group defined as
$$
G=\{h \in \mathrm{GL}(2g)_{/\Z} \,|\, ^th \psi h =  \psi, \, \nu(h) \in \mathbb{G}_{m/\Z}\}.
$$
Its derived group is $G^1=\mrm{Ker} \nu$. Let
$$
\mathfrak{H}_{g}^\pm=\{\tau \in \mathrm{M}_g(\C)\,|\,^t\tau=\tau, \pm \Im(\tau)\,\text{positive definite}\}
$$
be the disjoint union of Siegel upper and lower half planes. It follows for example from \cite{birkenhake-lange} Prop. 8.2.3 a) that the group $G(\R)$ acts transitively on $\mathfrak{H}_{g}^\pm$ by the formula
$$
\begin{pmatrix}
A & B\\
C & D\\
\end{pmatrix}.\tau=(A\tau+B)(C\tau+D)^{-1}.
$$
Now, let $h: \mathbb{S} \rightarrow G_{\R}$ be the morphism which induces on real points
\begin{equation} \label{morphism-h}
z=x+iy \mapsto \begin{pmatrix}
x^\star & y^\star\\
-y^\star & x^\star\\
\end{pmatrix}
\end{equation}
where, for $r \in \R$, we denote by $r^\star \in \mrm{M}(g, \R)$ the diagonal matrix whose all entries are equal to $r$. Then, the $G(\R)$-conjugacy class of $h$ is in bijection with $\mathfrak{H}_g^\pm$ by the unique $G(\R)$-equivariant map sending $h$ to $iI_2$ and the pair $(G_{\Q}, \mathfrak{H}_g^\pm)$ is a pure Shimura datum (see \cite{laumon} Lem. 2.1, which can be adapted from the case $g=2$ to the general case). Let $K(n)=\mrm{Ker}(G(\widehat{\Z}) \rightarrow G(\Z/n\Z))$ be the principal congruence subgroup of $G(\mathbb{A}_f).$ Define the arithmetic subgroup $\Gamma(n)$ of $G^1(\R)$ by
$$
\Gamma(n)=\mathrm{ker}\left( G^1(\Z) \rightarrow G^1(\Z/n\Z) \right).
$$
It follows from \cite{laumon} Prop. 3.2 that, as complex analytic varieties, we have
$$
\mathcal{S}(\C)=G(\Q)\backslash (\mathfrak{H}_g^\pm \times G(\mathbb{A}_f)/K(n))=\bigsqcup_{i \in (\Z/n\Z)^\times} \Gamma(n) \backslash \mathfrak{H}_{g}^+.
$$
Let us form the semi-direct product $\Lambda(n)=\Z^{2g} \rtimes \Gamma(n)$, where $\Gamma(n)$ acts on $\Z^{2g}$ by matrix multiplication. Then, as explained in \cite{birkenhake-lange} 8.7, 8.8, the group $\Lambda(n)$ acts on $\C^g \times \mathfrak{H}_g^+$ and we have a principally polarized holomorphic family of abelian varieties
\begin{equation} \label{univ-siegel}
\pi: \Lambda(n) \backslash (\C^g \times \mathfrak{H}_g^+) \rightarrow \Gamma(n) \backslash \mathfrak{H}_g^+
\end{equation}
with level $n$ structure. We will not need that this family is the restriction of the universal family to $\Gamma(n) \backslash \mathfrak{H}_g^+$ but only that this family is algebraic. This follows from \cite{borel} Thm. 3.10.

\subsection{Hilbert-Blumenthal varieties} \label{hb-var} Let $F$ be a totally real number field of degree $g$, let $\mathcal{O}$ be its ring of integers. Denote by $\mathrm{Tr}_{F/\Q}: F \rightarrow \Q$ the trace map and let 
$$
\mathcal{D}^{-1}=\{x\in F\,|\,\forall y \in \mathcal{O}, \mathrm{Tr}_{F/\Q}(xy) \in \Z\}
$$ be the inverse different. Define a group scheme $G'$ over $\Spec \Z$ by the cartesian square
\begin{equation} \label{groupeH}
\begin{CD}
G' @>>> \mathrm{Res}_{\mathcal{O}/\Z} \mathrm{GL}(2)\\
@VVV   @VV\mathrm{det}V\\
\mathbb{G}_m @>>> \mathrm{Res}_{\mathcal{O}/\Z} \mathbb{G}_m
\end{CD}
\end{equation}
where the lower map is the morphism which on $A$-valued points, for any ring $A$, is the morphism $A^\times \rightarrow (A \otimes \mathcal{O})^\times$ defined by $a \mapsto a \otimes 1$. Let 
$$
\mathfrak{H}^{g \pm}=\{\tau \in F \otimes_{\Q} \C \,|\, \pm \Im(\tau)\,\text{totally positive}\}.
$$ 
The group $G'(\R)$ acts on $\mathfrak{H}^{g \pm}$ via the embedding $G'(\R) \rightarrow \mathrm{GL}(2, \R)^g$ given by the $g$ embeddings $F \rightarrow \R$ and the pair $(G'_{\Q}, \mathfrak{H}^{g \pm})$ is a pure Shimura datum. Let $n \geq 3 $ be an integer. For any  prime ideal $\mathfrak{p}$ of $F$, let us denote by $\mathcal{O}_{\mathfrak{p}}$ the ring of integers of the $\mathfrak{p}$-adic completion of $F$ and by $\mathcal{D}_{\mathfrak{p}}$, resp. by $\mathcal{D}_{\mathfrak{p}}^{-1}$, the $\mathfrak{p}$-adic completion of $\mathcal{D}$, resp. $\mathcal{D}^{-1}$. Let $K'(n)$ be the compact open subgroup of $G'(\mbb{A}_f)$ defined as the product $K'(n)=\prod_{\mathfrak{p}} K'(n, \mathfrak{p})$ indexed by all prime ideals $\mathfrak{p}$ of $F$, where
$$
K'(n, \mathfrak{p})=\left\{ \begin{pmatrix}
a & b\\
c & d\\
\end{pmatrix} \in \mrm{GL}(2, F_\mathfrak{p}) \,|\, a,d \in 1+n\mathcal{O}_{\mathfrak{p}}, c \in n\mathcal{D}_{\mathfrak{p}}, b \in n \mathcal{D}^{-1}_{\mathfrak{p}} \right\}.
$$
We denote by $\mathcal{S}'$ the complex analytic Hilbert-Blumenthal variety 
$$
\mathcal{S}'=G'(\mbb{Q}) \backslash (\mathfrak{H}^{g \pm} \times G'(\mbb{A}_f)/K'(n))
$$
of level $K'(n)$. To describe a connected component of $\mathcal{S}'$, let us also consider the subgroup $\Gamma'(n)$ of $\mathrm{SL}(2, F)$ defined by
$$
\Gamma'(n)=\left\{ \begin{pmatrix}
a & b\\
c & d\\
\end{pmatrix} \in \mathrm{SL}(2, F)\,|\, a, d \in 1+n\mathcal{O}, c\in n \mathcal{D}, b \in n \mathcal{D}^{-1} \right\}.
$$
Then, the natural inclusion $\Gamma'(n) \backslash \mathfrak{H}^{g +} \rightarrow \mathcal{S}'$ identifies $\Gamma'(n) \backslash \mathfrak{H}^{g +}$ to a connected component of $\mathcal{S}'$. Let us form the semi-direct product $\Lambda'(n)=(\mathcal{D}^{-1} \oplus \mathcal{O}) \rtimes \Gamma'(n)$ where $\Gamma'(n)$ acts on $\mathcal{D}^{-1} \oplus \mathcal{O}$ by right matrix multiplication. Then $\Lambda'(n)$ acts on $\C^g \times \mathfrak{H}^g$ and we obtain an abelian scheme
\begin{equation} \label{univ-hb}
\pi': \Lambda'(n) \backslash (\C^g \times \mathfrak{H}^{g +}) \rightarrow \Gamma'(n) \backslash \mathfrak{H}^{g +}.
\end{equation}
Here, we refer the reader to \cite{blottiere} 2.2 for more details.

\subsection{A modular embedding} Let us explain how to map these Hilbert-Blumenthal varieties to the Siegel varieties defined above. Let $\sigma_1, \ldots, \sigma_g$ be the $g$ embeddings $F \rightarrow \R$ and let $(e_1, \ldots, e_g)$ be a basis of $\mathcal{O}$ over $\Z$.  As the map $F \times F \rightarrow \Q, (x, y) \mapsto \mathrm{Tr}_{F/\Q}(xy)$ is a non-degenerate bilinear form, we can define the dual basis $(e_1^*, \ldots, e_g^* )$ of $\mathcal{D}^{-1}$.  Let $R$ and $R'$ be the matrices 
$$
R=(\sigma_i(e_j))_{1 \leq i, j \leq g}, R'=(\sigma_i(e_j^*))_{1 \leq i, j \leq g}.
$$
They verify the identity $R'={}^t \! R^{-1}$. In the next proposition, we use the following notation: if $A$ is a $\Q$-algebra and $r \in A \otimes_{\Q} F$, we denote by $r^\star$ the diagonal $(A \otimes_{\Q} \R)$-valued matrix $\mrm{diag}((1 \otimes \sigma_1)(r), \ldots, (1 \otimes \sigma_g)(r))$. 

\begin{pro} \label{morph-Sh-data} The pair $(\overline{\iota}, \iota)$ where $\overline{\iota}: G'_{\Q} \rightarrow G_{\Q}$ is defined on $A$-valued points by
$$
\overline{\iota} \left( \begin{pmatrix}
a & b\\
c & d\\
\end{pmatrix} \right)=
\begin{pmatrix}
R' & 0\\
0 & R\\
\end{pmatrix}^{-1}
\begin{pmatrix}
a^\star & b^\star\\
c^\star & d^\star\\
\end{pmatrix}\begin{pmatrix}
R' & 0\\
0 & R\\
\end{pmatrix}
$$
and where $\iota: \mathfrak{H}^{g\,\pm} \rightarrow \mathfrak{H}_g^\pm$ is defined by
$$
\iota(\tau)=R'^{-1}\tau^\star R
$$
is a morphism of Shimura data
$$
(G'_{\Q},  \mathfrak{H}^{g\,\pm}) \rightarrow (G_{\Q},  \mathfrak{H}_g^\pm).
$$
\end{pro}

\begin{proof}
The only fact that is not obvious is that, for any $\Q$-algebra $A$, the morphism $\overline{\iota}$ maps $G'_{\Q}(A)$ to $G_{\Q}(A)$. To prove this point, let us regard $\mathcal{D}^{-1} \oplus \mathcal{O}$ as embedded in $\R^g \oplus \R^g$ via the $g$ embeddings $F \rightarrow \R$. Then, we have
$$
\mathcal{D}^{-1} \oplus \mathcal{O}=\begin{pmatrix}
R' & 0\\
0 & R\\
\end{pmatrix}\Z^g \oplus \Z^g
$$
where $\Z^g \oplus \Z^g$ is the standard lattice in $\R^g \oplus \R^g$. If we denote by $\psi': (\mathcal{D}^{-1} \oplus \mathcal{O})^{\oplus 2} \rightarrow \Z$ the symplectic form defined by $\psi'((x_1, y_1), (x_2, y_2))=\mathrm{Tr}_{F/\Q}(x_1 y_2-y_1 x_2),$
it follows from an easy computation that
$$
\begin{pmatrix}
{}^t \! R' & 0\\
0 & {}^t \! R\\
\end{pmatrix} \psi' \begin{pmatrix}
R' & 0\\
0 & R\\
\end{pmatrix}=\psi
$$
where $\psi$ is the symplectic form (\ref{symplectic}). As a consequence, the morphism $\overline{\iota}$ is well defined.
\end{proof}

Let 
$$
\mathrm{SL}(\mathcal{D}^{-1} \oplus \mathcal{O})=\left\{ \begin{pmatrix}
a & b\\
c & d\\
\end{pmatrix} \in \mathrm{SL}(2, F)\,|\, a, d \in \mathcal{O}, c\in \mathcal{D}, b \in \mathcal{D}^{-1} \right\}.
$$
Then $\mathrm{SL}(\mathcal{D}^{-1} \oplus \mathcal{O})$ sends $\mathcal{D}^{-1} \oplus \mathcal{O}$ to itself by right matrix multiplication. As a consequence $\overline{\iota}$ sends  $\mathrm{SL}(\mathcal{D}^{-1} \oplus \mathcal{O})$ to $G^1(\Z)$ and $\Gamma'(n)$ to $\Gamma(n)$. In particular, the pair $(\overline{\iota}, \iota)$ induces a holomorphic map
\begin{equation} \label{mod-embedding}
{\iota}: \Gamma'(n) \backslash \mathfrak{H}^{g +} \rightarrow \Gamma(n) \backslash \mathfrak{H}_g^+.
\end{equation}
Taking the sum of the different maps $\iota$ over all connected components we obtain a complex analytic morphism
\begin{equation} \label{mod-embedding'}
\iota: \mathcal{S}' \rightarrow \mathcal{S},
\end{equation}
which is in fact algebraic because it is induced by a morphism of Shimura data (\cite{pink} Prop. 11.10). The proof of the following result is easy and left to the reader.

\begin{pro} \label{carre-cartesien} The square in the category of $\C$-schemes
$$
\begin{CD}
\Lambda'(n) \backslash (\C^g \times \mathfrak{H}^{g +}) @>  f_ {R' R}\times {\iota} >> \Lambda(n) \backslash (\C^g \times \mathfrak{H}_g^+)\\
@V\pi' VV                                                                                                                    @V\pi VV\\                                                                                         
\Gamma'(n) \backslash \mathfrak{H}^{g +} @>{\iota}>> \Gamma(n) \backslash \mathfrak{H}_g^+,
\end{CD}
$$
where the map $f_{R' R}: \C^g=\R^g \oplus \R^g \rightarrow \C^g=\R^g \oplus \R^g$ is the map given by left multiplication by $\begin{pmatrix}
R' & 0\\
0 & R\\
\end{pmatrix}^{-1}$, is cartesian.
\end{pro}

\subsection{The Baily-Borel compactifications} In this section, let us denote by $G/\Q$ an arbitrary linear algebraic reductive group, which underlies a pure Shimura datum $(G, \mathfrak{H})$ (\cite{pink} Def. 2.1). We will only be interested in the case where $G$ is the group $\mrm{GSp}(2g)$ or the group denoted by $G'$ in the diagram (\ref{groupeH}). In this section, we wish to briefly recall the construction of the Baily-Borel compactification of the Shimura variety attached to $(G, \mathfrak{H})$ and the stratification of its boundary. We follow the presentation of \cite{burgos-wildeshaus} 1.\\

The Shimura varieties attached to $(G, \mathfrak{H})$ are indexed by compact open subgroups $K$ of $G(\mathbb{A}_f)$. Let $K \subset G(\mathbb{A}_f)$ be such a subgroup, which we assume to be neat (see \cite{pink} 0.6 for a definition of neatness). Then, the set of complex points of the corresponding variety $M^K(G, \mathfrak{H})$ over $\C$ is given by
$$
M^K(G, \mathfrak{H})(\C)=G(\Q) \backslash (\mathfrak{H} \times G(\mathbb{A}_f)/K).
$$
In order to describe the Baily-Borel compactification $M^K(G, \mathfrak{H})^*$ of $M^K(G, \mathfrak{H})$ recall that for any admissible parabolic subgroup $Q$ of $G$ (\cite{pink} Def. 4.5) there is associated a canonical normal subgroup $P_1$ of $Q$ (\cite{pink} 4.7). There is a finite collection of rational boundary components $(P_1, \mathfrak{X}_1)$ (\cite{pink} 4.11) which are mixed Shimura data. Denote by $W_1$ the unipotent radical of $P_1$ and by $(G_1, \mathfrak{H}_1)$ the quotient of $(P_1, \mathfrak{X}_1)$ by $W_1$ (\cite{pink} Prop. 2.9). One defines
$$
\mathfrak{H}^*=\bigsqcup_{(P_1, \mathfrak{X}_1)} \mathfrak{H}_1
$$
where the disjoint union is indexed by all rational boundary components $(P_1, \mathfrak{X}_1)$ of $(G, \mathfrak{H})$. This set comes equipped with the Satake topology (\cite{pink} 6.2) as well as a natural action of the group $G(\Q)$ (\cite{pink} 4.16). Let
$$
M^K(G, \mathfrak{H})^*(\C)=G(\Q) \backslash (\mathfrak{H}^* \times G(\mathbb{A}_f)/K)
$$
equipped with the quotient topology. By \cite{pink} 8.2, this space can be canonically identified with the space of $\C$-valued points of a normal projective variety $M^K(G, \mathfrak{H})^*$ over $\C$ containing $M^K(G, \mathfrak{H})$ as a Zariski dense open subset. The stratification of $\mathfrak{H}^*$ induces a stratification of $M^K(G, \mathfrak{H})^*$ as follows. Fix an admissible parabolic subgroup $Q$ of $G$ and let $(P_1, \mathfrak{X}_1)$, $W_1$ and
$$
p: (P_1, \mathfrak{X}_1) \rightarrow (G_1, \mathcal{H}_1)=(P_1, \mathfrak{X}_1)/W_1
$$
be as above. Let $g \in G(\mbb{A}_f)$, let $K'=gKg^{-1}$ and let $K_1=P_1(\mbb{A}_f) \cap K'$. We have the following natural morphisms
$$
\begin{CD}
M^{p(K_1)}(G_1, \mathfrak{H}_1)(\C)=
G_1(\Q) \backslash (\mathfrak{H}_1 \times G_1(\mathbb{A}_f)/p(K_1))\\
@AAA \\
P_1(\Q) \backslash (\mathfrak{H}_1 \times P_1(\mathbb{A}_f)/K_1) \\
@VVV \\
M^K(G, \mathfrak{H})^*(\C)=G(\Q) \backslash (\mathfrak{H}^* \times G(\mathbb{A}_f)/K)
\end{CD}
$$
where the first map is induced by the map $(x, h) \mapsto (x, p(h))$ and the second map is induced by $(x, h) \mapsto (x, hg)$. The first map is an isomorphism of complex analytic variteties. Hence, we obtain a morphism 
\begin{equation} \label{strate}
M^{p(K_1)}(G_1, \mathfrak{H}_1) \rightarrow M^K(G, \mathfrak{H})^*
\end{equation}
which depends on the rational boundary component $(P_1, \mathfrak{X}_1)$ and on $g$. When $(P_1, \mathfrak{X}_1)$ and $g$ vary, the images of the morphisms (\ref{strate}) form a stratification of $M^K(G, \mathfrak{H})^*$.\\

Let us make explicit some of the notions introduced above in the case where $(G, \mathfrak{H})$ is the Shimura datum $(G', \mathfrak{H}^{g \pm})$ and in the case where $(G, \mathfrak{H})$ is the Shimura datum $(\mrm{GSp}(2g), \mathfrak{H}_g^\pm)$. Up to conjugacy, the unique admissible parabolic subgroup of the group $G'$ defined by diagram (\ref{groupeH}) is the standard Borel subgroup $B'$ of $G'$, i.e. the intersection with $G'$ of the subgroup of upper triangular matrices in $\mathrm{Res}_{F/\Q} \mathrm{GL}(2)$. The canonical normal subgroup $P'$ of $B'$ is the intersection with $G'$ of matrices of the shape $
\begin{pmatrix}
* & *\\
0 & 1\\
\end{pmatrix}$
in $\mathrm{Res}_{F/\Q} \mathrm{GL}(2)$. Let $W'$ be the unipotent radical of $P'$ and let $p': P' \rightarrow P'/W'=\mbb{G}_m$ be the canonical projection. The quotient of the rational boundary component $(P', \mathfrak{X}')$ by $W'$ is the Shimura datum $(\mathbb{G}_m, \mathfrak{H}_0^\pm)$ defined in \cite{pink} Example 2.8 as follows: let $k: \mbb{S} \rightarrow \mbb{G}_{m \R}$ be the morphism inducing $z \mapsto z \overline{z}$ on real points and let $\mathfrak{H}_0^\pm$ be a set with two elements endowed with the unique non-trivial action of $\mbb{G}_m(\R)$ which factors through $\pi_0(\mbb{G}_m(\R))$. We map $\mathfrak{H}_0^\pm$ to $\mrm{Hom}(\mbb{S}_{\C}, \mbb{G}_{m \C})$ by the constant map equal to $k$. We have the diagram
\begin{equation} \label{diagram-HB}
\begin{CD}
\mathcal{S}' @>j'>> \mathcal{S}^{' *} @<i'<< \partial \mathcal{S}'
\end{CD}
\end{equation}
where $j'$ denotes the open immersion from $\mathcal{S}'$ to its Baily-Borel compactification and $i'$ denotes the complementary closed embedding. As the Shimura varieties attached to $(\mathbb{G}_m, \mathfrak{H}_0^\pm)$ are of dimension zero (\cite{pink} Example 3.16), the boundary $\partial \mathcal{S}'$ is of dimension zero.\\

The standard admissible parabolic subgroups $Q$ of $G$ and their canonical normal subgroups $P_1$ are described in \cite{morel} 1.2. In particular, for any rational boundary component $(P_1, \mathfrak{X}_1)$ of $(G, \mathfrak{H}_g^\pm)$, if $W_1$ denotes the unipotent radical of $P_1$, we know that the quotient Shimura datum $(P_1, \mathfrak{X}_1)/W_1$ is either $(\mathbb{G}_m, \mathfrak{H}_0^\pm)$ or $(\mathrm{GSp}(2r), \mathfrak{H}_r^\pm)$ for an integer $1 \leq r \leq g$. Similarly as above, we have the diagram
\begin{equation} \label{diagram-siegel}
\begin{CD}
\mathcal{S} @>j>> \mathcal{S}^{*} @<i<< \partial \mathcal{S}
\end{CD}
\end{equation}
where $j$ denotes the open immersion from $\mathcal{S}$ to its Baily-Borel compactification and $i$ denotes the complementary closed embedding. Let us denote by $\partial \mathcal{S}_0$ the stratum of the boundary whose underlying rational boundary component $(P_1, \mathfrak{X}_1)$ is such that the identity $(P_1, \mathfrak{X}_1)/W_1=(\mathbb{G}_m, \mathfrak{H}_0^\pm)$ holds. Then $\partial \mathcal{S}_0$ has dimension zero.  

\begin{lem} \label{technical} For any integer $r \geq 1$, there exists no morphism of Shimura data
$$
(\mathbb{G}_m, \mathfrak{H}_0^\pm) \rightarrow (\mathrm{GSp}(2r), \mathfrak{H}_r^\pm).
$$
\end{lem}

\begin{proof} Let us assume the contrary. According to Def. \ref{morphism-Shimura}, this implies that there exists a morphism $\Psi: \mbb{G}_m \rightarrow \mrm{GSp}(2r)$ such that the morphism $h_r: \mbb{S} \rightarrow \mrm{GSp}(2r)_{\R}$ defined by (\ref{morphism-h}) factors as $h_r=\Psi_{\R} \circ k$. Let us denote by $\{w_1, w_2, \ldots, w_k\}$ the weights, counted with multiplicities, of $\Psi$ on the vector space $\Q^{2r}$ on which $\mrm{GSp}(2r)$ acts naturally. Let $w: \mbb{G}_{m \R} \rightarrow \mbb{S}$ be the weight cocharacter, which is the cocharacter inducing the inclusion $\R^\times \subset \C^\times$ on real points. Then,  the composition $h_r \circ w$ has even weights $\{2w_1, 2w_2, \ldots, 2w_k\}$, which contraditcs the fact that the morphism $h_r \circ w$ induces $x \mapsto \mrm{diag}(x, \ldots, x)$ on real points.
\end{proof}

\begin{pro} \label{map-boundary} There exists a continuous map $\tilde{\iota}: \mathcal{S}^{' *} \rightarrow \mathcal{S}^*$ which is part of the following diagram with cartesian squares in the category of topological spaces
\begin{equation} \label{diagram}
\begin{CD}
\mathcal{S}' @>j' >> \mathcal{S}^{' *} @<i' << \partial \mathcal{S}'\\
@V\iota VV                  @V\tilde{\iota}VV                                   @V\partial \iota VV\\
\mathcal{S} @>j>>    \mathcal{S}^*      @<i<< \partial \mathcal{S}_0.\\
\end{CD}
\end{equation}
\end{pro}

\begin{proof} According to Prop. \ref{morph-Sh-data}, we have a morphism of Shimura data 
$$
(G'_{\Q},  \mathfrak{H}^{g\,\pm}) \rightarrow (G_{\Q},  \mathfrak{H}_g^\pm).
$$
Hence, the existence of a continuous map $\tilde{\iota}^{an}: (\mathcal{S}^{' *})^{an} \rightarrow (\mathcal{S}^*)^{an}$ such that the above left hand square is cartesian follows from the functoriality of rational boundary components (\cite{pink} 4.16) and from the construction of the Baily-Borel compactification as described above. We want to show that the right hand square is cartesian. Let $(P_1, \mathfrak{X}_1)$ be the rational boundary component of $(G, \mathfrak{H}_g^\pm)$ associated to $(P', \mathfrak{X}')$ by the construction of \cite{pink} 4.16 and let $W_1$ be the unipotent radical of $P_1$. By the construction of $\tilde{\iota}$ that we just scketched, we need to show that $
(P_1, \mathfrak{X}_1)/W_1=(\mbb{G}_m, \mathfrak{H}_0^\pm).$
By \cite{pink} 4.16 p. 87, we have a morphism of Shimura data
 $$
(P', \mathfrak{X}')/W'= (\mbb{G}_m, \mathfrak{H}_0^\pm) \rightarrow (P_1, \mathfrak{X}_1)/W_1.
$$
Hence the statement follows from Lem. \ref{technical} and from the construction of the Baily-Borel compactification.
\end{proof}

\section{The residue of Eisenstein classes of Siegel varieties}

Let $g \geq 1$ and $n \geq 3$ be two integers and let $\mathcal{S}$ be the Siegel variety of genus $g$ and level $n$ (see section \ref{siegel-var}). We have the universal abelian scheme $\mathcal{A}/\mathcal{S}$. Let $F$ be a totally real number field of degree $g$ and let $\mathcal{S}'$ be the associated Hilbert-Blumenthal variety of level $n$ (see section \ref{hb-var}). 

\begin{lem} \label{dc} Let $F \in D^b_c(\mathcal{S}, \Q)$. Then, the morphisms of diagram (\ref{diagram}) induce a commutative diagram
\begin{equation} \label{diagram-coho}
\begin{CD}
H^{2g-1}(\mathcal{S}, F) @>>> H^{0}(\partial \mathcal{S}_0, i^*R^{2g-1}j_* F)\\
@VVV                                              @VVV\\
H^{2g-1}(\mathcal{S}', \iota^* F) @>>> H^{0}(\partial \mathcal{S}', i'^*R^{2g-1}j'_* \iota^*F).\\
\end{CD}
\end{equation}
\end{lem}

\begin{proof} Let us consider the diagram (\ref{diagram}). Applying $j_*$ to the adjunction map $F \rightarrow \iota_* \iota^* F$ and using the commutativity of $(\ref{diagram})$, we obtain the map $j_* F \rightarrow \tilde{\iota}_* j'_* \iota^* F$. Composing with the adjunction $1 \rightarrow i_* i^*$, and using the proper base change theorem $i^* \tilde{\iota}_*=(\partial \iota)_* i^{' *}$, we obtain the commutative diagram
$$
\begin{CD}
j_* F @>>> i_*i^*j_* F\\
@VVV           @VVV\\
\tilde{\iota}_* j'_* \iota^* F @>>> i_* (\partial \iota)_* i^{'*} j'_* \iota^* F.
\end{CD}
$$
The proof of the statement of the lemma now follows by appliying the functor $R^{2g-1}p_*$, where $p: \mathcal{S}^* \rightarrow \Spec \C$ is the structural morphism, and using the fact that $\partial \mathcal{S}_0$ and $\partial \mathcal{S}'$ are of dimension zero.
\end{proof}

\begin{thm} \label{main}  Let $k \geq 2$ be an even integer. There exists a $n$-torsion section $t: \mathcal{S} \rightarrow \mathcal{A}$ such that the image of the Eisenstein class 
$$
\mathrm{Eis}^{gk}(t) \in H^{2g-1}(\mathcal{S}, \mathrm{Sym}^{gk} \mathcal{H}(g))
$$
under the map $H^{2g-1}(\mathcal{S}, \mathrm{Sym}^{gk} \mathcal{H}(g)) \rightarrow H^{0}(\partial \mathcal{S}_0, i^*R^{2g-1}j_* \mathrm{Sym}^{gk} \mathcal{H}(g))$
is non-zero.
\end{thm}

\begin{proof} Let us consider a connected component $\Gamma(n) \backslash \mathfrak{H}_g^+$ of $\mathcal{S}$. As explained at the end of section \ref{siegel-var}, we have the principally polarized abelian scheme 
$$
\pi: \Lambda(n) \backslash (\C^g \times \mathfrak{H}_g^+) \rightarrow \Gamma(n) \backslash \mathfrak{H}_g^+
$$
Hence, by the modular interpretation of $\mathcal{S}$, we have a cartesian square
$$
\begin{CD}
\Lambda(n) \backslash (\C^g \times \mathfrak{H}_g^+) @>>> \mathcal{A}\\
@V\pi VV                                                                                          @VVV\\
\Gamma(n) \backslash \mathfrak{H}_g^+ @>>> \mathcal{S}.
\end{CD}
$$
By Cor. \ref{chgt-de-base-eisenstein}, it is enough to show that there exists $s$ a $n$-torsion section of $\pi$ such that $\mathrm{Eis}^{gk}(s)$ has non-zero image under the map $$H^{2g-1}(\mathcal{S}, \mathrm{Sym}^{gk} \mathcal{H}(g)) \rightarrow H^{0}(\partial \mathcal{S}_0, i^*R^{2g-1}j_* \mathrm{Sym}^{gk} \mathcal{H}(g)).$$ Let us consider a connected component $\Gamma'(n) \backslash \mathfrak{H}^{g +}$ of $\mathcal{S}'$, which is the base of the abelian scheme
$$
\pi': \Lambda'(n) \backslash (\C^g \times \mathfrak{H}^{g +}) \rightarrow \Gamma'(n) \backslash \mathfrak{H}^{g +}.
$$
According to Prop. \ref{carre-cartesien}, we have a cartesian square
$$
\begin{CD}
\Lambda'(n) \backslash (\C^g \times \mathfrak{H}^{g +}) @>  f_ {R' R}\times {\iota} >> \Lambda(n) \backslash (\C^g \times \mathfrak{H}_g^+)\\
@V\pi' VV                                                                                                                    @V\pi VV\\                                                                                         
\Gamma'(n) \backslash \mathfrak{H}^{g +} @>{\iota}>> \Gamma(n) \backslash \mathfrak{H}_g^+.
\end{CD}
$$
Hence, by Cor. \ref{chgt-de-base-eisenstein} and thanks to the commutative diagram (\ref{diagram-coho}), it is enough to show that there exists a torsion section $t'$ of $\pi'$ such that the associated Eisenstein class $\mathrm{Eis}^{gk}(t')$ has non-zero image under the map
$$H^{2g-1}(\mathcal{S}', \mathrm{Sym}^{gk} \mathcal{H}'(g)) \rightarrow H^{0}(\partial \mathcal{S}', i^*R^{2g-1}j_* \mathrm{Sym}^{gk} \mathcal{H}'(g)).$$
This is a direct consequence of \cite{blottiere} Thm. 5.2 and Cor. 5.4.
\end{proof}

\begin{rems} It would be interesting to know if the image of the $k$-th Eisenstein class $\mrm{Eis}^k(t)$ under the residue map $H^{2g-1}(\mathcal{S}, \mathrm{Sym}^{k} \mathcal{H}(g)) \rightarrow H^{0}(\partial \mathcal{S}_0, i^*R^{2g-1}j_* \mathrm{Sym}^{k} \mathcal{H}(g))$ is zero or not, without assuming that $k$ is even and divisible by $g$. What about the residue on higher dimensional strata of the Baily-Borel compactification?
\end{rems}

\begin{cor} Let $k \geq 2$ be an even integer. Let $\mathcal{A}$ be the universal abelian scheme over $\mathcal{S}$ and let $\mathcal{A}^{gk}$ be the $gk$-th fold fiber product over $\mathcal{S}$. Then, the higher regulator map
$$
H^{gk+2g-1}_\mathcal{M}(\mathcal{A}^{gk}, \Q(gk+g)) \rightarrow H^{gk+2g-1}(\mathcal{A}^{gk}, \Q(gk+g))
$$
from motivic cohomology to Betti cohomology is non-zero.
\end{cor}

\begin{proof} By Liebermann's trick, the space $H^{gk+2g-1}(\mathcal{A}^{gk}, \Q(gk+g))$ contains the space $H^{2g-1}(\mathcal{S}, \mathrm{Sym}^{gk} \mathcal{H}(g))$. Hence, the statement follows from Thm. \ref{main} and the fact that Eisenstein classes belong to the image of the higher regulator map, according to \cite{kings1}.
\end{proof}

\end{document}